\documentclass[11pt]{amsart}

\usepackage{amscd}
\usepackage{amsmath, amssymb}
\usepackage{amsfonts}
\newcommand{\de}{\partial}
\newcommand{\db}{\overline{\partial}}
\newcommand{\ddbar}{\sqrt{-1} \partial \overline{\partial}}
\newcommand{\ov}[1]{\overline{#1}}
\newcommand{\mn}{\sqrt{-1}}
\newcommand{\tr}[2]{\mathrm{tr}_{#1}{#2}}
\newcommand{\ti}[1]{\tilde{#1}}

\newcommand{\ve}{\varepsilon}
\newcommand{\re}{\textrm{Re}}

\renewcommand{\leq}{\leqslant}
\renewcommand{\geq}{\geqslant}
\renewcommand{\le}{\leqslant}
\renewcommand{\ge}{\geqslant}

\begin{document}
\newtheorem{claim}{Claim}
\newtheorem{theorem}{Theorem}[section]
\newtheorem{conjecture}[theorem]{Conjecture}
\newtheorem{lemma}[theorem]{Lemma}
\newtheorem{corollary}[theorem]{Corollary}
\newtheorem{proposition}[theorem]{Proposition}
\newtheorem{question}[theorem]{question}
\newtheorem{defn}[theorem]{Definition}
\theoremstyle{definition}
 \newtheorem{remark}[theorem]{Remark}
\numberwithin{equation}{section}

\newenvironment{example}[1][Example]{\addtocounter{remark}{1} \begin{trivlist}
\item[\hskip
\labelsep {\bfseries #1  \thesection.\theremark}]}{\end{trivlist}}
\title[Complex Monge-Amp\`ere with gradient term]
{The complex Monge-Amp\`ere equation \\ with a gradient term}

\author[V. Tosatti]{Valentino Tosatti}
\author[B. Weinkove]{Ben Weinkove}
\address{Department of Mathematics, Northwestern University, 2033 Sheridan Road, Evanston, IL 60208}
\thanks{Partially supported by NSF grants DMS-1610278 (V.T.) and DMS-1709544 (B.W.). Part of this work was done
while the first-named author was visiting the Center for Mathematical Sciences and Applications at Harvard University, which he thanks for the hospitality.}
\dedicatory{Dedicated to Professor D.H. Phong on the occasion of his 65th birthday}

\begin{abstract} We consider the complex Monge-Amp\`ere equation with an additional linear gradient term inside the determinant.  We prove existence and uniqueness of solutions to this equation on compact Hermitian manifolds.
\end{abstract}

\maketitle

\section{Introduction}

Let $M$ be a compact complex manifold of complex dimension $n$.  When $M$ admits a K\"ahler metric $g=(g_{i\ov{j}})$, Yau \cite{Ya} proved the now classic result that the complex Monge-Amp\`ere equation
\begin{equation} \label{kma}
\det (g_{i\ov{j}} +   u_{i\ov{j}}) = e^{F}\det (g_{i\ov{j}}) , \quad (g_{i\ov{j}} + u_{i\ov{j}})>0,
\end{equation}
admits a unique solution $u$ with $\sup_M u=0$, as long as $F$ is normalized so that  $(e^F-1)$ has zero integral.  Equivalently, one can prescribe the volume form  of a K\"ahler metric within a given K\"ahler class.

Yau's result has been extended and built on in various ways.  Modulo adding a constant to $F$, the equation (\ref{kma}) can be solved for $g$ Hermitian (by work of Cherrier \cite{Ch} and the authors \cite{TW2}, see also \cite{GL, TW1}) and  for $g$ almost Hermitian (Chu-Tosatti-Weinkove \cite{CTW}).
Fu-Wang-Wu \cite{FWW1, FWW2} considered the Monge-Amp\`ere equation obtained by taking the determinant of the $(n-1, n-1)$ form
$$\omega^{n-1} + \sqrt{-1} \partial \ov{\partial} u \wedge \omega^{n-2}.$$
This is the natural equation on compact manifolds associated to Harvey-Lawson's notion of $(n-1)$-plurisubharmonicity \cite{HL}, and was solved for $\omega$ Hermitian by the authors \cite{TW3, TW5}.
Building on this work, Sz\'ekelyhidi-Tosatti-Weinkove \cite{STW} proved existence of solutions for Monge-Amp\`ere equation associated to
$$\omega^{n-1} + \sqrt{-1} \partial \ov{\partial} u \wedge \omega^{n-2} + L(x, \nabla u),$$
for the specific first order term
\begin{equation} \label{fot}
L(x, \nabla u) = \textrm{Re}(\sqrt{-1} \partial u \wedge \ov{\partial} \omega^{n-2})
\end{equation} introduced by Popovici \cite{Po} and independently in \cite{TW5}.    This yielded a solution of the Gauduchon conjecture \cite{Ga} on the existence of Gauduchon metrics with prescribed volume form.  The proof in \cite{STW} makes careful use of the specific form of this first order term term $L(x, \nabla u)$. See also \cite{GN,FGZ,Sz,Zh} for related follow-up work.

Other nonlinear equations involving gradient terms arise naturally by motivations from mathematical physics,  including the Fu-Yau equation \cite{FY} and its extensions by Phong-Picard-Zhang \cite{PPZ1, PPZ2, PPZ3}.    In particular, the paper \cite{PPZ1} considers the complex Hessian equations $$(\chi(z,u) + \sqrt{-1} \partial \ov{\partial} u)^k \wedge \omega^{n-k} = \psi(z, u,\nabla u) \omega^n$$ where gradient terms appear on the right hand side.

In light of these results, it is natural to consider
fully nonlinear equations in terms of the metric $$\ti{\omega}=\omega +  \sqrt{-1} \partial \ov{\partial} u + L(x,\nabla u),$$
for $L$ a  linear term involving the gradient of $u$.
Indeed, this study was initiated recently by R. Yuan \cite{Yuan}.  However the family of equations he deals with  includes the Monge-Amp\`ere equation
$\ti{\omega}^n = e^{F} \omega^n$ only in the case of complex  dimension $n=2$ \cite[Corollary 1.5]{Yuan}.  The current paper settles the case  $n>2$ left open by Yuan.

More precisely, let $(M, g)$ be a compact Hermitian manifold of complex dimension $n$.  By analogy to (\ref{fot}), we consider the term
$$L(x,\nabla u) = \mn a\wedge\db u-\mn\ov{a}\wedge\de u$$
where $a$ is a smooth $(1,0)$-form.  Indeed, this is the most general term of the form $\alpha \wedge \partial u+ \beta\wedge \ov{\partial}u$ for $1$-forms $\alpha$ and $\beta$, which is also real and of type $(1,1)$.  In local coordinates, we may write $L(x,\nabla u) = \sqrt{-1} (a_i u_{\ov{j}} + a_{\ov{j}} u_i) dz^i \wedge d\ov{z}^j$, where  $a=a_i dz^i$ and $a_{\ov{i}}=\ov{a_i}$.

We prove the following:

\begin{theorem}\label{main}  Given $F \in C^{\infty}(M)$ and a smooth $(1,0)$ form $a$ on $M$, there exists a unique pair $(u, b)$ with $u\in C^{\infty}(M)$ and $b \in \mathbb{R}$ satisfying the equation
\begin{equation} \label{maf}
\begin{split}
& \det (g_{i\ov{j}} +  a_i u_{\ov{j}} + a_{\ov{j}} u_i + u_{i\ov{j}}) = e^{F+b}\det (g_{i\ov{j}}) , \\
&  \mathrm{with\ \ } (\tilde{g}_{i\ov{j}}) := (g_{i\ov{j}} + a_i u_{\ov{j}} + a_{\ov{j}} u_i + u_{i\ov{j}})>0, \quad \mathrm{ and } \ \sup_Mu=0.
\end{split}
\end{equation}
\end{theorem}

The case $n=2$ is due to Yuan \cite{Yuan}.  We also remark that Zhang \cite{Zha} proved a uniform gradient estimate for a class of equations which includes (\ref{maf}).

We can rewrite \eqref{maf} in coordinate-free notation by letting $$\ti{\omega}:=\omega+\mn a\wedge\db u-\mn\ov{a}\wedge\de u+\ddbar u>0,$$
be the new Hermitian metric whose volume form equals
$$\ti{\omega}^n=e^{F+b}\omega^n.$$
\begin{remark}
As an aside, note that if we choose $a$ to be a holomorphic $1$-form, then we can write
\begin{equation}\label{aep}
\ti{\omega}=\omega+\de\ov{\gamma}+\db\gamma,
\end{equation}
where $\gamma$ is the $(1,0)$ form given by
$$\gamma=-\mn\left(ua+\frac{\de u}{2}\right).$$
In this case, if we also have that $\de\db\omega=0$ (which when $n=2$ is the Gauduchon condition \cite{Ga0}), then $\omega$ defines a cohomology class in Aeppli cohomology, and \eqref{aep} shows that the metric $\ti{\omega}$ also satisfies $\de\db\ti{\omega}=0$ and lies in the same Aeppli cohomology class.
\end{remark}

The outline of our proof is as follows.  We begin by proving \emph{a priori} estimates for solutions of (\ref{maf}).
  In Section \ref{sectionzero}, we establish a uniform $L^{\infty}$ bound for $u$, with an approach that uses the Aleksandrov-Bakelman-Pucci estimate.  In Section \ref{sectionsecond} we give an estimate on the second derivatives $\sqrt{-1}\partial\ov{\partial}u$ of $u$ in terms of the first derivatives, using a maximum principle argument involving the largest eigenvalue $\lambda_1$ of the metric $\tilde{g}$. The particular quantity we use for the maximum principle is
  $$Q=\log \lambda_1 +  \frac{ | \partial u|^2_g}{\sup_M | \partial u|^2_g +1} + e^{-A u},$$
  for a large constant $A$.  This differs (and in many cases is simpler) than the quantities used in the literature mentioned above.
     To overcome the fact that the eigenvalue $\lambda_1$ is not differentiable in general, we choose to use a viscosity argument (adapted from \cite{BCD}, and hinted to in \cite{Sz}), which to our knowledge is new in this Hermitian setting.  Finally, in Section \ref{sectionproof}, we complete the proof of Theorem \ref{main}:  we apply a standard blow-up argument to obtain the first order estimate and then standard theory gives the higher order estimates.  Given the $C^{\infty}$ \emph{a priori }estimates, the existence follows from a fairly standard continuity argument and uniqueness is a consequence of the maximum principle.
 
 Instead of using a blow-up argument, the gradient estimate can be obtained directly by a maximum principle argument, as shown in an earlier work of Zhang \cite[Remark 2]{Zha} (see also the related works \cite{B3,FGZ,Yuan}).  We thank the referee for pointing out the reference \cite{Zha}, of which we were not aware  when we completed the first version of this article.
  
  \medskip
\noindent
{\bf Acknowledgments. }  Both authors owe many thanks to Professor Phong, to whom this article is dedicated.  His mathematical wisdom and insights are an inspiration to us.  Happy birthday Phong!

\section{Zero order estimate} \label{sectionzero}

Let $u,F\in C^\infty(M)$ and $a\in \Lambda^{1,0}M$ satisfy
\begin{equation} \label{pde}
\begin{split}
& \det (g_{i\ov{j}} +  a_i u_{\ov{j}} + a_{\ov{j}} u_i + u_{i\ov{j}}) = e^{F}\det (g_{i\ov{j}}) \\
&  \quad  (\tilde{g}_{i\ov{j}}):= (g_{i\ov{j}} + a_i u_{\ov{j}} + a_{\ov{j}} u_i + u_{i\ov{j}})>0,
\end{split}
\end{equation}
with $\sup_M u=0$.  We will write $\tilde{\omega}$ for the $(1,1)$ form associated to the metric $\tilde{g}_{i\ov{j}}$.

We prove a uniform estimate for $u$.

\begin{theorem}\label{0th}
%Let $u,F\in C^\infty(M)$ and $a\in \Lambda^{1,0}M$ satisfy
%\begin{equation} \label{pde}
%\det (g_{i\ov{j}} +  a_i u_{\ov{j}} + a_{\ov{j}} u_i + u_{i\ov{j}}) = e^{F}\det (g_{i\ov{j}}), \quad g_{i\ov{j}} + a_i u_{\ov{j}} + a_{\ov{j}} u_i + u_{i\ov{j}}>0,
%\end{equation}
%with $\sup_M u=0$.
There is a constant $C$ that depends only on $\sup_M|F|$, $\sup_M |a|_g$, and on the geometry of $(M,g)$ such that
\begin{equation}\label{li}
\sup_M|u|\leq C.
\end{equation}
\end{theorem}
\begin{proof}
We employ the Aleksandrov-Bakelman-Pucci estimate, whose usage  for the complex Monge-Amp\`ere equation originated in work of Cheng-Yau (see \cite{Be}), and was more recently revisited by B\l ocki \cite{B,B2} and Sz\'ekelyhidi \cite{Sz}.  We follow \cite{CTW,Sz, TW4}.

First, we observe that
\begin{equation}\label{l1}
\int_M(-u)\omega^n\leq C,
\end{equation}
for a uniform constant $C$. Indeed, let
$$H(u)=\Delta_g u +\tr{\omega}{(\mn a\wedge\db u -\mn\ov{a}\wedge\de u)}=\tr{g}{\ti{g}}-n\geq -n,$$
where $\Delta_gu=\tr{\omega}{\ddbar u}=\frac{n\ddbar u\wedge\omega^{n-1}}{\omega^n}$ is the complex Laplacian of $g$.
Since the kernel of $H$ consists of just constants, a classical argument of Gauduchon \cite{Ga0} (cf. \cite[Theorem 2.2]{CTW}) shows that there is a smooth function $v$ such that
\begin{equation}\label{hha}
\int_M H(\psi) e^v\omega^n=0,
\end{equation}
for all smooth functions $\psi$. We then define a new Hermitian metric $\hat{\omega}=e^{v/(n-1)}\omega$. Its operator $\hat{H}$, defined in the same way
\begin{equation}\label{hhat}
\hat{H}(\psi)=\Delta_{\hat{g}} \psi +\tr{\hat{\omega}}{(\mn a\wedge\db \psi -\mn\ov{a}\wedge\de \psi)},
\end{equation}
satisfies
\begin{equation}\label{ddd}
\hat{H}(u)=e^{-v/(n-1)}H(u)\geq -C,
\end{equation}
and now we have
\begin{equation}\label{dd}
\int_M \hat{H}(\psi) \hat{\omega}^n=0,
\end{equation}
for all $\psi$. We may then use the Green's function for $\hat{H}$ (with respect to the metric $\hat{\omega}$), to deduce the uniform $L^1$ bound for $u$ in \eqref{l1} by the exact same argument as in \cite[Proof of Theorem 2.1]{TW5}. Briefly, standard theory gives us a Green's function $G(x,y)$, normalized to have zero integral, which has a uniform lower bound and such that
$$\psi(x)=\frac{1}{\int_M\hat{\omega}^n}\int_M \psi\hat{\omega}^n-\int_M \hat{H}(\psi)(y)G(x,y)\hat{\omega}^n(y),$$
holds for all $\psi$ and all $x\in M$. Thanks to \eqref{dd} we can add a uniform constant to $G$ to make it nonnegative, while preserving the same Green's formula, and we then apply this to $u$ with $x$ a point where $u(x)=0$, so that from \eqref{ddd} and the lower bound for $G$ we easily deduce \eqref{l1}.

Next, we  promote the $L^1$ bound \eqref{l1} to the $L^\infty$ bound \eqref{li} using ABP, as in \cite[Proposition 3.1]{CTW} and \cite{Sz, TW4}. Let $x_0\in M$ be a point where $u$ achieves its infimum $I=\inf_M u$, and fix a coordinate unit ball $B$ centered at $x_0$. In this ball, let $v=u+\ve |x|^2$, where $\ve>0$ will be a uniform constant to be chosen later. We have $\inf_{\de B}v\geq v(0)+\ve$, so \cite[Proposition 10]{Sz} gives us that
\begin{equation}\label{dddd}
\ve^{2n}\leq C\int_P\det(D^2v),
\end{equation}
for a universal constant $C$, where
$$P = \{ x \in B \ | \ |Dv(x)|<\ve/2, \textrm{ and } v(y) \ge v(x) + Dv(x)\cdot(y-x) \ \forall y \in B \}.$$
Given now any $x\in P$, we have $D^2v(x)\geq 0$ and $|Du(x)|\leq 5\ve/2$ so at $x$
$$\mn a\wedge \db u-\mn\ov{a}\wedge \de u +\ddbar u\geq -C\ve\omega,$$
for a uniform constant $C$, therefore if we choose $\ve$ sufficiently small (but uniformly bounded away from zero), we get
$$\ti{\omega}(x)\geq\frac{1}{2}\omega(x),$$
and from the Monge-Amp\`ere equation \eqref{pde} we deduce
$$\ti{\omega}(x)\leq C\omega(x),$$
from which
$$\ddbar u(x)\leq C\omega(x),$$
and so $0\leq \ddbar v(x)\leq C\omega(x)$. But a simple linear algebra inequality (using that $(D^2v(x))\geq 0$) gives
$$\det(D^2v(x))\leq C \det(v_{i\ov{j}})^2(x)\leq C,$$
which together with \eqref{dddd} gives
$$|P|\geq C^{-1},$$
where $|P|$ denotes the Lebesgue measure. For all $x\in P$ we have
$$v(x)\leq v(0)+\frac{\ve}{2}=I+\frac{\ve}{2},$$
and we may assume that $I+\frac{\ve}{2}<0$, so
$$C^{-1}\leq|P|\leq\frac{\int_P(-v)}{|I+\frac{\ve}{2}|}\leq\frac{C}{|I+\frac{\ve}{2}|},$$
using the $L^1$ bound \eqref{l1}, which proves \eqref{li}.
\end{proof}

\section{Second order estimate} \label{sectionsecond}

In this section we prove a bound on $\sqrt{-1}\partial \ov{\partial}u$ in terms of a bound on the square of the first derivative of $u$.  This estimate takes the same form as the Hou-Ma-Wu estimate \cite{HMW} for the complex Hessian equations (see also the later works \cite{CTW, Sz, STW, TW3, TW5}) although here the quantity to which we apply the maximum principle is slightly simpler.

\begin{theorem}\label{2nd}
Let $u,F\in C^\infty(M)$ and $a\in \Lambda^{1,0}M$ satisfy \eqref{pde}, with $\sup_M u=0$. Then there is a constant $C$ that depends only on $\sup_M|u|$, $\|a\|_{C^2(M)}$, $\|F\|_{C^2(M)}$ and on the geometry of $(M,g)$ such that
$$\sup_M | \sqrt{-1}\partial \ov{\partial} u|_g \le C(1+\sup_M|\de u|^2_g).$$
\end{theorem}
\begin{proof}
Define the linearized operator $L$ by
\begin{equation} \label{defnL}
Lv = \tilde{g}^{i\ov{j}} (v_{i\ov{j}} +  a_i v_{\ov{j}} + a_{\ov{j}} v_i) = \tilde{g}^{i\ov{j}}v_{i\ov{j}} + 2\textrm{Re} \left( \tilde{g}^{i\ov{j}} a_{\ov{j}} v_i \right).
\end{equation}
Observe that
\begin{equation} \label{Lu}
Lu = \tilde{g}^{i\ov{j}} (\tilde{g}_{i\ov{j}} - g_{i\ov{j}}) = n - \tr{\tilde{g}}{g}.
\end{equation}

Let $\lambda_1 \ge \lambda_2 \ge \cdots \ge \lambda_n>0$ be the eigenvalues of $\tilde{g}_{i\ov{j}}$ with respect to $g$.  We consider the quantity
$$Q= \log \lambda_1 + \varphi( | \partial u|^2_g) + \psi(u),$$
where we define
$$\varphi(s) = \frac{s}{K}, \ s\ge 0, \quad \textrm{and} \quad \psi(t) =  e^{-At}, \ t \le 0,$$
with
$$K = \sup_M | \partial u|^2_g+1,$$
and $A>0$ to be determined.
Note that we have
$$ - \psi' \ge A>0, \quad \psi'' = -A\psi'.$$

 We assume that $Q$ achieves its maximum at $x_0\in M$.  It suffices  to show that at $x_0$, we have $\lambda_1 \le CK$ for a uniform $C$.  Hence in what follows we may assume without loss of generality that $\lambda_1$ is large compared to $K$.  We will calculate at the point $x_0$ using coordinates for which $g$ is the identity and $\tilde{g}$ is diagonal with entries $\tilde{g}_{i\ov{i}}=\lambda_i$ for $i=1, \ldots, n$.

Since $\lambda_1$ may not be smooth at $x_0$, we
define a  smooth function $f$ on $M$ by (cf. \cite[Proof of Theorem 6]{BCD})
\begin{equation} \label{defnf}
Q(x_0) = \log f + \varphi(| \partial u|^2_g) + \psi(u),
\end{equation}
where the right hand side of (\ref{defnf}) is evaluated at a general point of $M$.
Observe that $f$ satisfies
\begin{equation}\label{fcondition}
f \ge \lambda_1 \quad \textrm{on } M, \quad f = \lambda_1 \quad \textrm{at } x_0.
\end{equation}

We have the following lemma, which is a complex version of \cite[Lemma 5]{BCD}.  Here and in the sequel, we  use $\nabla_i$ or simply lower indices (after commas, when needed to avoid confusion) to denote covariant derivatives with respect to the Chern connection of $g$.

\begin{lemma} \label{lemmaviscosity}
Let $\mu$ denote the multiplicity of the largest eigenvalue of $\tilde{g}$ at $x_0$, so that $\lambda_1=\cdots = \lambda_{\mu} > \lambda_{\mu+1} \ge \cdots \ge \lambda_n$.  Then at $x_0$, for each $i$ with $1\le i \le n$,
\begin{equation} \label{fd}
\tilde{g}_{k\ov{\ell}, i} = f_i g_{k\ov{\ell}} , \quad \textrm{for } 1\le k, \ell \le \mu,
\end{equation}
and
\begin{equation} \label{sd}
 f_{i\ov{i}} \ge  \tilde{g}_{1\ov{1}, i\ov{i}} + \sum_{q>\mu} \frac{ | \tilde{g}_{q\ov{1},i}|^2 + | \tilde{g}_{q\ov{1}, \ov{i}}|^2}{\lambda_1-\lambda_q}.
\end{equation}
\end{lemma}
\begin{proof}  The proof only uses the fact that $f$ is smooth and satisfies (\ref{fcondition}).
For a smooth vector field $V = V^k \frac{\partial}{\partial z^k}$ defined in a neighborhood of $x_0$, we consider the function
$$h = \tilde{g}_{k\ov{\ell}} V^k \ov{V^{\ell}} - f g_{k\ov{\ell}} V^k \ov{V^{\ell}},$$
which is nonpositive.
For any choice of $V$ with $V^k(x_0) =0$ for $k > \mu$ we have $h(x_0)=0$ and hence
 $h$ has a local maximum at $x_0$.

For (\ref{fd}), choose $V$ with $V^k(x_0)=0$ for $k>\mu$ and $$\nabla_i V^k(x_0)=0 = \nabla_{\ov{i}} V^k(x_0), \quad \textrm{for } k \le \mu.$$
Then at $x_0$,
\[
\begin{split}
0 = h_i = {}
%& \nabla_i \tilde{g}_{k\ov{\ell}} V^k \ov{V^{\ell}} - \partial_i f g_{k\ov{\ell}} V^k \ov{V^{\ell}} + \tilde{g}_{k\ov{\ell}} \nabla_i V^k \ov{V^{\ell}} \\ {} & \mbox{} + \tilde{g}_{k\ov{\ell}} V^k \ov{ \nabla_{\ov{i}} V^{\ell}} - \partial_i f g_{k\ov{\ell}} \nabla_i V^k \ov{V^{\ell}} - \partial_i f g_{k\ov{\ell}} V^k \ov{\nabla_{\ov{i}} V^{\ell}} \\ =
{} &\tilde{g}_{k\ov{\ell},i} V^k \ov{V^{\ell}} -  f_i g_{k\ov{\ell}} V^k \ov{V^{\ell}},
\end{split}
\]
and (\ref{fd}) follows since we can choose $V^k(x_0)$ for $k \le \mu$ to be whatever we like.

For (\ref{sd}) we choose $V$ with $V(x_0) = \frac{\partial}{\partial z^1}$ and
$$\nabla_i V^q(x_0) = \left\{ \begin{array}{ll} 0, \quad & q \le \mu \\ \frac{ \tilde{g}_{1\ov{q},i}}{\lambda_1-\lambda_q}, \quad & q >\mu \end{array} \right.$$
and
$$\nabla_{\ov{i}} V^q(x_0) = \left\{ \begin{array}{ll} 0, \quad & q \le \mu \\ \frac{ \tilde{g}_{1\ov{q}, \ov{i}}}{\lambda_1-\lambda_q}, \quad & q >\mu. \end{array} \right.$$
Then  at $x_0$,
\begin{equation} \label{calc1}
\begin{split}
0 \ge h_{i\ov{i}}  = {} &  \tilde{g}_{1\ov{1}, i\ov{i}} -  f_{i\ov{i}} +  \tilde{g}_{k\ov{\ell},i} (\nabla_{\ov{i}} V^k) \ov{V^{\ell}} +  \tilde{g}_{k\ov{\ell},i} V^k \ov{\nabla_i V^{\ell}}
+  \tilde{g}_{k\ov{\ell}, \ov{i}} (\nabla_i V^k) \ov{V^{\ell}} \\  & \mbox{}  +  \tilde{g}_{k\ov{\ell}, \ov{i}} V^k \ov{\nabla_{\ov{i}} V^{\ell}} + \tilde{g}_{k\ov{\ell}} \nabla_i V^k \ov{\nabla_i V^{\ell}} + \tilde{g}_{k\ov{\ell}} \nabla_{\ov{i}} V^k \ov{\nabla_{\ov{i}} V^{\ell}} \\
{} & \mbox{} - f g_{k\ov{\ell}} \nabla_i V^k \ov{\nabla_i V^{\ell}} - f g_{k\ov{\ell}} \nabla_{\ov{i}} V^k \ov{\nabla_{\ov{i}} V^{\ell}},
\end{split}
\end{equation}
noting that terms of the type
$f_i g_{k\ov{\ell}} (\nabla_{\ov{i}} V^k) \ov{V^{\ell}}$ vanish by definition of $V$ and
$$\tilde{g}_{k\ov{\ell}} (\nabla_{\ov{i}} \nabla_i V^k) \ov{V^{\ell}} - f g_{k\ov{\ell}}  (\nabla_{\ov{i}} \nabla_i V^k) \ov{V^{\ell}}=0=\tilde{g}_{k\ov{\ell}} V^k \ov{\nabla_i \nabla_{\ov{i}}V^{\ell}} - f g_{k\ov{\ell}}  V^k \ov{\nabla_i \nabla_{\ov{i}} V^{\ell}}$$
since $fg_{1\ov{1}} = \lambda_1 = \tilde{g}_{1\ov{1}}$ at $x_0$. Continuing from (\ref{calc1}), using the definition of $V$,
\[
\begin{split}
0 \ge {} &  \tilde{g}_{1\ov{1}, i\ov{i}} -  f_{i\ov{i}} + 2\sum_{q > \mu} \frac{ | \tilde{g}_{q\ov{1},i}|^2}{\lambda_1-\lambda_q} + 2\sum_{q > \mu} \frac{ |  \tilde{g}_{q\ov{1},\ov{i}}|^2}{\lambda_1-\lambda_q} \\
{} & \mbox{} + \sum_{q>\mu} \lambda_q \frac{ | \tilde{g}_{1\ov{q},i}|^2}{(\lambda_1-\lambda_q)^2}  +  \sum_{q>\mu} \lambda_q \frac{ |  \tilde{g}_{1\ov{q},\ov{i}}|^2}{(\lambda_1-\lambda_q)^2} \\
{} & \mbox{} - \lambda_1 \sum_{q>\mu}  \frac{ |  \tilde{g}_{1\ov{q},i}|^2}{(\lambda_1-\lambda_q)^2}  -  \lambda_1 \sum_{q>\mu} \frac{ |  \tilde{g}_{1\ov{q},\ov{i}}|^2}{(\lambda_1-\lambda_q)^2} \\
= {} &   \tilde{g}_{1\ov{1},i\ov{i}} - f_{i\ov{i}} + \sum_{q>\mu} \frac{ |  \tilde{g}_{q\ov{1},i}|^2 + |  \tilde{g}_{q\ov{1},\ov{i}}|^2}{\lambda_1-\lambda_q},
\end{split}
\]
as required.
\end{proof}

Differentiating (\ref{pde}) we obtain
\begin{equation} \label{1d}
\tilde{g}^{i\ov{i}} \tilde{g}_{i\ov{i}, k} = \tilde{g}^{i\ov{i}} (u_{i\ov{i}k} + a_{i,k} u_{\ov{i}} + a_i u_{k\ov{i}} +a_{\ov{i},k} u_i + a_{\ov{i}} u_{ik})= F_k,
\end{equation}
where here and henceforth we are computing at the point $x_0$.
Differentiating again, and setting $k=1$,
\begin{equation} \label{F11}
\tilde{g}^{i\ov{i}} \tilde{g}_{i\ov{i},1\ov{1}} - \tilde{g}^{i\ov{i}} \tilde{g}^{j\ov{j}} \tilde{g}_{i\ov{j}, 1}  \tilde{g}_{j\ov{i},\ov{1}}  = F_{1\ov{1}}.
\end{equation}

Now apply $\nabla_i$ to  the defining equation (\ref{defnf}) of $f$ to obtain
\begin{equation} \label{Qi}
0 =  \frac{f_i}{\lambda_1} + \varphi' \left( u_p u_{\ov{p}i} +  u_{pi} u_{\ov{p}} \right) + \psi' u_i.
\end{equation}
Next apply the operator $L$, as defined in (\ref{defnL}), to the defining equation of $f$ to obtain,
\begin{equation} \label{LQ}
\begin{split}
0 = {} &  \frac{\tilde{g}^{i\ov{i}}f_{i\ov{i}}}{\lambda_1} - \frac{\tilde{g}^{i\ov{i}} | f_i|^2}{\lambda_1^2} %+ \varphi''  \tilde{g}^{i\ov{i}} \left|  u_p u_{\ov{p}i} +  u_{pi} u_{\ov{p}} \right|^2
 + \varphi' \sum_{p} \tilde{g}^{i\ov{i}} \left( | u_{p\ov{i}}|^2 + |u_{pi}|^2 \right) \\{}& \mbox{} + \varphi'  \tilde{g}^{i\ov{i}} (u_{pi\ov{i}} u_{\ov{p}} + u_{\ov{p} i \ov{i}} u_p)
  + \psi''   \tilde{g}^{i\ov{i}} | u_i|^2 + \psi' (n- \tr{\tilde{g}}{g}) \\ & \mbox{}+
2\re \left(  \tilde{g}^{i\ov{i}} a_{\ov{i}} \frac{f_i}{\lambda_1} \right)
+ 2\varphi' \re \left(  \tilde{g}^{i\ov{i}} a_{\ov{i}} \left(  u_p u_{\ov{p}i} +  u_{pi} u_{\ov{p}} \right) \right),
\end{split}
\end{equation}
where we have made use of  (\ref{Lu}).
We wish to compare $\sum_i \tilde{g}^{i\ov{i}} f_{i\ov{i}}$ and $\sum_i \tilde{g}^{i\ov{i}} \tilde{g}_{i\ov{i},1\ov{1}}$.  From Lemma \ref{lemmaviscosity},
\begin{equation} \label{lambda1d}
f_i = \tilde{g}_{11,\ov{i}}, \ \textrm{and } f_{i\ov{i}} \ge \tilde{g}_{1\ov{1}, i\ov{i}} + \sum_{q >\mu} \frac{ | \tilde{g}_{1\ov{q}, i}|^2 + | \tilde{g}_{q\ov{1}, i}|^2}{\lambda_1- \lambda_q}.
\end{equation}
To compare $\tilde{g}_{11,i\ov{i}}$ and $\tilde{g}_{i\ov{i},1\ov{1}}$ we first compute, using $T^{k}_{ij}$ and $R_{k\ov{\ell}i}^{\ \ \ \, p}$ to denote the torsion and Chern curvature tensors of $g$ respectively (see for example \cite{TW5}),
\begin{equation} \label{u4}
\begin{split}
u_{i\ov{i}1\ov{1}} = {} & u_{i\ov{i}\ov{1}1} + R_{1\ov{1}i}^{\ \ \ \, p} u_{p\ov{i}} - R_{1\ov{1} \ \, \ov{i}}^{\ \ \, \ov{q}} u_{i\ov{q}} \\
= {} & u_{i\ov{1}\ov{i}1} + R_{1\ov{1}i}^{\ \ \ \, p} u_{p\ov{i}} - R_{1\ov{1} \ \, \ov{i}}^{\ \ \, \ov{q}} u_{i\ov{q}} + \nabla_1 \ov{T^q_{i1}} u_{i\ov{q}} + \ov{T^q_{i1}} u_{i\ov{q}1} \\
= {} & u_{\ov{1}i1\ov{i}} + R_{1\ov{1}i}^{\ \ \ \, p} u_{p\ov{i}} - R_{1\ov{1} \ \, \ov{i}}^{\ \ \, \ov{q}} u_{i\ov{q}} + \nabla_1 \ov{T^q_{i1}} u_{i\ov{q}} + \ov{T^q_{i1}} u_{i\ov{q}1}  \\ {} & + R_{1\ov{i} \ \, \ov{1}}^{\ \ \, \ov{q}} u_{\ov{q}i} - R_{1\ov{i}i}^{\ \ \  p} u_{\ov{1}p} \\
= {} & u_{1\ov{1} i\ov{i}}  + R_{1\ov{1}i}^{\ \ \ \, p} u_{p\ov{i}} - R_{1\ov{1} \ \, \ov{i}}^{\ \ \, \ov{q}} u_{i\ov{q}} + \nabla_1 \ov{T^q_{i1}} u_{i\ov{q}} + \ov{T^q_{i1}} u_{i\ov{q}1}  \\
{} &  + R_{1\ov{i} \ \, \ov{1}}^{\ \ \, \ov{q}} u_{\ov{q}i} - R_{1\ov{i}i}^{\ \ \  p} u_{\ov{1}p} + \nabla_{\ov{i}} T^q_{i1} u_{\ov{1}q} + T^q_{i1} u_{\ov{1}q \ov{i}},
\end{split}
\end{equation}
where for the second inequality and fourth inequalities, we used the formulae
\begin{equation} \label{u3}
u_{j\ov{\ell}  \ov{k}} - u_{j \ov{k}\ov{\ell}}= \ov{T^q_{\ell k}} u_{j\ov{q}}, \quad u_{\ov{j} \ell k} - u_{\ov{j} k \ell} = T^q_{\ell k} u_{\ov{j}q}.
\end{equation}

%$$u_{i\ov{i}1\ov{1}} = u_{1\ov{1} i\ov{i}} + u_{p\ov{i}} R_{1\ov{1}i}^{\ \ \ \, p} - u_{p\ov{1}} R_{i\ov{i} 1}^{\ \ \ \, p}.$$
From (\ref{u4}) and the definition of $\tilde{g}_{i\ov{j}}$,
\[
\begin{split}
\tilde{g}^{i\ov{i}} \tilde{g}_{1\ov{1},i\ov{i}} = {} & \tilde{g}^{i\ov{i}} \tilde{g}_{i\ov{i},1\ov{1}} + \tilde{g}^{i\ov{i}} \{ u_{1\ov{1}i\ov{i}} - u_{i\ov{i}1\ov{1}}
 + a_{1,i\ov{i}} u_{\ov{1}} - a_{i,1\ov{1}} u_{\ov{i}} \\
{} & + a_{\ov{1},i\ov{i}} u_1 - a_{\ov{i},1\ov{1}} u_i + a_{1,i} u_{\ov{1}\ov{i}} - a_{i,1} u_{\ov{i}\ov{1}} + a_{1,\ov{i}} u_{\ov{1} i} - a_{i,\ov{1}} u_{\ov{i}1} \\
{} & + a_{\ov{1},i} u_{1\ov{i}} - a_{\ov{i},1}u_{i\ov{1}} + a_{\ov{1}, \ov{i}} u_{1i} - a_{\ov{i},\ov{1}} u_{i1} + a_1 u_{\ov{1}i\ov{i}} - a_i u_{\ov{i}1 \ov{1}} \\
{} & + a_{\ov{1}} u_{1i\ov{i}} - a_{\ov{i}} u_{i1\ov{1}} \} \\
\ge{} & \tilde{g}^{i\ov{i}} \tilde{g}_{i\ov{i},1\ov{1}} + \tilde{g}^{i\ov{i}} \left( \ov{T^q_{1i}} u_{i\ov{q}1} + T^q_{1i} u_{\ov{1}q\ov{i}} + a_1 u_{\ov{1}i\ov{i}} - a_i u_{\ov{i}1 \ov{1}} + a_{\ov{1}} u_{1i\ov{i}} - a_{\ov{i}} u_{i1\ov{1}} \right) \\
& -  \sum_{p} \tilde{g}^{i\ov{i}} \left( |u_{p\ov{i}}|^2 + |u_{pi}|^2 \right) - C( \tr{\tilde{g}}{g} )( \tr{g}{\tilde{g}}),
\end{split}
\]
where for the last line we used the assumption that $K \le \lambda_1 \le \tr{g}{\tilde{g}}$, and the uniform lower bound of $\tr{g}{\tilde{g}}$ which follows from our equation (\ref{pde}).

Next, observe that
\begin{equation} \label{comm1}
\begin{split}
& u_{i\ov{j} k} = u_{k\ov{j}i} + T^p_{ik} u_{\ov{j}p}= u_{ki\ov{j}}  + T^p_{ik} u_{\ov{j}p}- u_p R_{i\ov{j} k}^{\ \ \ \, p}.
\end{split}
\end{equation}
Then, using this and (\ref{1d}),
\[
\begin{split}
\tilde{g}^{i\ov{i}} (a_1 u_{\ov{1} i \ov{i}} + a_{\ov{1}} u_{1i\ov{i}}) = {} & 2 \re\left( \tilde{g}^{i\ov{i}} a_{\ov{1}} u_{1i \ov{i}} \right) - a_{1} u_{\ov{q}}  \tilde{g}^{i\ov{i}} R_{i\ov{i}\  \, \ov{1}}^{\  \, \, \ov{q}} \\
= {} & 2 \re \left( \tilde{g}^{i\ov{i}} a_{\ov{1}} \left( u_{i\ov{i} 1}  -T^p_{i1} u_{\ov{i}p} + u_p R_{i\ov{i} 1}^{\ \ \ \, p} \right) \right) - a_{1} u_{\ov{q}}  \tilde{g}^{i\ov{i}} R_{i\ov{i}\  \, \ov{1}}^{\  \, \, \ov{q}} \\
= {} & 2 \re \left( a_{\ov{1}} F_{1} - \tilde{g}^{i\ov{i}} a_{\ov{1}} (T^p_{i1} u_{\ov{i}p} - u_pR_{i\ov{i}1}^{\ \ \ \, p}  \right. \\ {} & \left. \mbox{} + a_{i,1} u_{\ov{i}} + a_i u_{1\ov{i}} + a_{\ov{i},1} u_i + a_{\ov{i}} u_{i1})
\right)   - a_{1} u_{\ov{q}}  \tilde{g}^{i\ov{i}} R_{i\ov{i}\  \, \ov{1}}^{\  \, \, \ov{q}}.
\end{split}
\]
We also have
\begin{equation*}
\begin{split}
\lefteqn{\tilde{g}^{i\ov{i}} ( \ov{T^q_{1i}} u_{i\ov{q}1} + T^q_{1i} u_{\ov{1}q\ov{i}})} \\= {} & 2 \re \left( \tilde{g}^{i\ov{i}} \ov{T^q_{1i}} u_{1\ov{q}i}   \right) + \tilde{g}^{i\ov{i}} \ov{T^q_{1i}} T^p_{1i} u_{\ov{q}p} \\
= {} & 2 \re \left( \tilde{g}^{i\ov{i}} \ov{T^q_{1i}} \left( \tilde{g}_{1\ov{q},i} - a_{1,i} u_{\ov{q}} - a_1 u_{\ov{q}i} - a_{\ov{q},i} u_1 - a_{\ov{q}} u_{1i}   \right)  \right) \\ {} & \mbox{} + \tilde{g}^{i\ov{i}} \ov{T^q_{1i}} T^p_{i1} u_{\ov{q}p}.
\end{split}
\end{equation*}
Combining the above with (\ref{F11}) gives
\begin{equation} \label{gii11ii}
\begin{split}
\tilde{g}^{i\ov{i}} \tilde{g}_{1\ov{1},i\ov{i}}
\ge{} & \tilde{g}^{i\ov{i}} \tilde{g}^{j\ov{j}} \tilde{g}_{i\ov{j}, 1}  \tilde{g}_{j\ov{i},\ov{1}}  + 2 \re \left( \tilde{g}^{i\ov{i}} \ov{T^q_{1i}}  \tilde{g}_{1\ov{q},i} \right) - \tilde{g}^{i\ov{i}} \{   a_i u_{\ov{i}1 \ov{1}}  + a_{\ov{i}} u_{i1\ov{1}} \} \\
& -  2\sum_{i,p} \tilde{g}^{i\ov{i}} \left( |u_{p\ov{i}}|^2 + |u_{pi}|^2 \right) - C (\tr{\tilde{g}}{g} )( \tr{g}{\tilde{g}} ).
\end{split}
\end{equation}
Next, using again Lemma \ref{lemmaviscosity},
\begin{equation} \label{2Relambda1i}
\begin{split}
2\re \left( \tilde{g}^{i\ov{i}} a_{\ov{i}} \frac{f_i}{\lambda_1} \right) = {} & 2\re \left( \tilde{g}^{i\ov{i}} a_{\ov{i}} \frac{\tilde{g}_{1\ov{1},i}}{\lambda_1} \right)  \\=
{} &  \frac{\tilde{g}^{i\ov{i}}}{\lambda_1} \left( a_{\ov{i}} u_{i1\ov{1}}  + a_i u_{\ov{i} 1 \ov{1}} +a_{\ov{i}}T^p_{1i} u_{\ov{1}p}- a_{\ov{i}} u_p R_{1\ov{1}i}^{\ \ \ \, p} + a_i \ov{T^q_{1i}} u_{1\ov{q}}\right) \\
{} & + 2\re \left( \frac{\tilde{g}^{i\ov{i}}}{\lambda_1} a_{\ov{i}} \{ a_{1,i} u_{\ov{1}} + a_1 u_{\ov{1}i} + a_{\ov{1},i}u_1 +a_{\ov{1}}u_{1i} \} \right) \\
\ge {}  & \frac{\tilde{g}^{i\ov{i}}}{\lambda_1} \left( a_{\ov{i}} u_{i1\ov{1}}  + a_i u_{\ov{i} 1 \ov{1}} \right) - \frac{1}{\lambda_1}  \sum_{p}   \tilde{g}^{i\ov{i}} \left( |u_{p\ov{i}}|^2 + |u_{pi}|^2 \right)   \\ & \mbox{}  - C \tr{\tilde{g}}{g},
\end{split}
\end{equation}
and note that the terms involving three derivatives of $u$ exactly match those from (\ref{gii11ii}), after multiplying by $-1/\lambda_1$.

Now from (\ref{1d}) we have,
\[
\begin{split}
\tilde{g}^{i\ov{i}} u_{ i\ov{i}p} u_{\ov{p}} = {}&  F_p u_{\ov{p}} - \tilde{g}^{i\ov{i}} a_{i,p} u_{\ov{i}} u_{\ov{p}} - \tilde{g}^{i\ov{i}} a_i u_{\ov{i}p} u_{\ov{p}} - \tilde{g}^{i\ov{i}} a_{\ov{i},p} u_i u_{\ov{p}} - \tilde{g}^{i\ov{i}} a_{\ov{i}} u_{ip} u_{\ov{p}}.
\end{split}
\]
Hence, making use of (\ref{comm1}), and recalling that $\varphi' =1/K$,
\begin{equation} \label{upii}
\begin{split}
\lefteqn{ \varphi' \tilde{g}^{i\ov{i}} (u_{pi\ov{i}} u_{\ov{p}} + u_{\ov{p}i\ov{i}} u_p)} \\ =  {} & \varphi' \tilde{g}^{i\ov{i}} \left( u_{i\ov{i}p} u_{\ov{p}} + u_{\ov{i}i \ov{p}}u_p +   u_r u_{\ov{p}} R_{i\ov{i}p}^{\ \ \ \, r} -  T^r_{ip}u_{\ov{p}} u_{\ov{i}r} + \ov{T^q_{pi}} u_pu_{i\ov{q}}  \right) \\  =  {} &
2\varphi' \re \left(F_p u_{\ov{p}} - \tilde{g}^{i\ov{i}} a_{i,p} u_{\ov{i}} u_{\ov{p}} - \tilde{g}^{i\ov{i}} a_i u_{\ov{i}p} u_{\ov{p}}  - \tilde{g}^{i\ov{i}} a_{\ov{i},p} u_i u_{\ov{p}} - \tilde{g}^{i\ov{i}} a_{\ov{i}} u_{ip} u_{\ov{p}}  \right) \\ & \mbox{}  +  \varphi' \tilde{g}^{i\ov{i}} u_r u_{\ov{p}} R_{i\ov{i}p}^{\ \ \ \, r}  - 2 \varphi' \re\left( \tilde{g}^{i\ov{i}} T^r_{ip}u_{\ov{p}} u_{\ov{i}r} \right) \\
\ge {} & - \frac{\varphi'}{4} \sum_p \tilde{g}^{i\ov{i}} (|u_{p\ov{i}}|^2 + |u_{pi}|^2 ) - C \tr{\tilde{g}}{g}.
\end{split}
\end{equation}
We also have
\begin{equation} \label{phi2r}
2\varphi' \re \left(  \tilde{g}^{i\ov{i}} a_{\ov{i}} \left(  u_p u_{\ov{p}i} +  u_{pi} u_{\ov{p}} \right) \right) \ge  -  \frac{\varphi'}{4} \sum_p \tilde{g}^{i\ov{i}} (|u_{p\ov{i}}|^2 + |u_{pi}|^2 ) - C \tr{\tilde{g}}{g}.
\end{equation}

Combining (\ref{LQ}), (\ref{lambda1d}), (\ref{gii11ii}), (\ref{2Relambda1i}), (\ref{upii}) and (\ref{phi2r}) gives
\begin{equation} \label{tog}
\begin{split}
0 \ge {} & \frac{\tilde{g}^{i\ov{i}} \tilde{g}^{j\ov{j}} \tilde{g}_{i\ov{j},1} \tilde{g}_{j\ov{i}, \ov{1}}}{\lambda_1}  +  \sum_{q>\mu} \frac{ \tilde{g}^{i\ov{i}}( | \tilde{g}_{1\ov{q}, i}|^2 + | \tilde{g}_{q\ov{1}, i}|^2)}{\lambda_1(\lambda_1- \lambda_q)} -  \frac{\tilde{g}^{i\ov{i}} | \tilde{g}_{1\ov{1}, i}|^2}{\lambda_1^2} \\ {} & +  \frac{2\re \left( \tilde{g}^{i\ov{i}} \ov{T^q_{1i}}  \tilde{g}_{1\ov{q},i} \right)}{\lambda_1}  \mbox{} +  \left(\frac{1}{2} \varphi' - \frac{C}{\lambda_1} \right)\sum_{p} \tilde{g}^{i\ov{i}} \left( | u_{p\ov{i}}|^2 + |u_{pi}|^2 \right)  \\ & \mbox{}
  + \psi''   \tilde{g}^{i\ov{i}} | u_i|^2 + \psi' (n- \tr{\tilde{g}}{g}) - C \tr{\tilde{g}}{g} \\
%{} & + 2\varphi' \re \left(   \tilde{g}^{i\ov{i}} a_{\ov{i}} \left(   u_p u_{i\ov{p}} +   u_{pi} u_{\ov{p}} \right) \right),
\end{split}
\end{equation}
for  $C$ a universal constant (depending on $F$, $a$ etc).

We need to get a lower bound of
\begin{equation} \label{dontforget}
\frac{\tilde{g}^{i\ov{i}} \tilde{g}^{j\ov{j}} \tilde{g}_{i\ov{j},1} \tilde{g}_{j\ov{i}, \ov{1}}}{\lambda_1} -  \frac{\tilde{g}^{i\ov{i}} | \tilde{g}_{1\ov{1}, i}|^2}{\lambda_1^2}  \ge \sum_{i=2}^n \frac{\tilde{g}^{i\ov{i}} \tilde{g}_{i\ov{1}, 1} \tilde{g}_{1\ov{i},\ov{1}}}{\lambda_1^2} - \sum_{i=2}^n \frac{\tilde{g}^{i\ov{i}} | \tilde{g}_{1\ov{1},i}|^2}{\lambda_1^2},
\end{equation}
where we have discarded the terms with $j\neq 1$.
But note that
\[
\begin{split}
\tilde{g}_{i\ov{1},1} = {} & \tilde{g}_{1\ov{1},i} + \lambda_1  X_{1\ov{1}i},
\end{split}
\]
where $X_{1\ov{1}i}$ is defined by
$$X_{1\ov{1}i}:=\frac{1}{\lambda_1} \left(
 T^p_{i1} u_{\ov{1}p} + a_{i,1} u_{\ov{1}} + a_i u_{1\ov{1}}+a_{\ov{1},1} u_i - a_{1,i} u_{\ov{1}} - a_1 u_{i\ov{1}} - a_{\ov{1},i}u_1 + a_{\ov{1}} T_{1i}^k u_k\right),$$
and satisfies $|X_{1\ov{1}i}|\le C$ for a uniform $C$.  In the above, we used (\ref{comm1}) and the formula
$$u_{ij}-u_{ji} = T^k_{ji}u_k.$$

Then
\begin{equation} \label{uno}
\begin{split}
\sum_{i=2}^n \frac{\tilde{g}^{i\ov{i}} \tilde{g}_{i\ov{1}, 1} \tilde{g}_{1\ov{i},\ov{1}}}{\lambda_1^2}  \ge {} &  \sum_{i=2}^n  \frac{\tilde{g}^{i\ov{i}} | \tilde{g}_{1\ov{1},i}|^2}{\lambda_1^2} + 2\textrm{Re} \left( \sum_{i=2}^n \frac{\tilde{g}^{i\ov{i}} g_{1\ov{1}, i} \ov{X_{1\ov{1}i}}}{\lambda_1} \right).
\end{split}
\end{equation}
To deal with the second term, we use (\ref{Qi}) to compute
\begin{equation} \label{dos}
\begin{split}
\lefteqn{2 \textrm{Re}\left( \sum_{i=2}^n \frac{\tilde{g}^{i\ov{i}} \tilde{g}_{1\ov{1}, i} \ov{X_{1\ov{1}i}}}{\lambda_1}  \right)} \\ ={} & - 2 \textrm{Re} \left( \sum_{i=2}^n \tilde{g}^{i\ov{i}}(\varphi' (u_p u_{\ov{p}i} + u_{pi}u_{\ov{p}}) + \psi' u_i) \ov{X_{1\ov{1}i}} \right) \\
\ge {} & - \frac{\varphi'}{8} \sum_p \tilde{g}^{i\ov{i}}(|u_{p\ov{i}}|^2+|u_{pi}|^2)  - C\tr{\tilde{g}}{g}  + \psi' ( C \tilde{g}^{i\ov{i}}|u_i|^2 + \frac{1}{4} \tr{\tilde{g}}{g}),
\end{split}
\end{equation}
where we recall that $\psi'<0$.

Next we deal with the fourth term on the right hand side of (\ref{tog}).  From Lemma \ref{lemmaviscosity} we have $\tilde{g}_{1\ov{q},i}=0$ for $1<q\le \mu$ and hence
\begin{equation} \label{tres}
\begin{split}
 \frac{2\re \left( \tilde{g}^{i\ov{i}} \ov{T^q_{1i}}  \tilde{g}_{1\ov{q},i} \right)}{\lambda_1} = {} &  \frac{2\re \left( \tilde{g}^{i\ov{i}} \ov{T^1_{1i}}  \tilde{g}_{1\ov{1},i} \right)}{\lambda_1} + 2 \sum_{q>\mu}  \frac{\re \left( \tilde{g}^{i\ov{i}} \ov{T^q_{1i}}  \tilde{g}_{1\ov{q},i} \right)}{\lambda_1}
 \end{split}
 \end{equation}
 But using the same argument as in (\ref{dos}), replacing $|X_{1\ov{1}i}|\le C$ by $|T^1_{1i}|\le C$, we obtain
 \begin{equation} \label{tres2}
 \begin{split}
  \frac{2\re \left( \tilde{g}^{i\ov{i}} \ov{T^1_{1i}}  \tilde{g}_{1\ov{1},i} \right)}{\lambda_1} \ge {} & - \frac{\varphi'}{8}  \sum_p \tilde{g}^{i\ov{i}}(|u_{p\ov{i}}|^2+|u_{pi}|^2) - C\tr{\tilde{g}}{g}   \\ {} & + \psi' ( C \tilde{g}^{i\ov{i}}|u_i|^2 + \frac{1}{4} \tr{\tilde{g}}{g}).
  \end{split}
  \end{equation}
  On the other hand we have
  \begin{equation} \label{tres3}
  \begin{split}
  2 \sum_{q>\mu}  \frac{\re \left( \tilde{g}^{i\ov{i}} \ov{T^q_{1i}}  \tilde{g}_{1\ov{q},i} \right)}{\lambda_1} \ge {} & - \sum_{q>\mu} \frac{\tilde{g}^{i\ov{i}} |\tilde{g}_{1\ov{q},i}|^2}{\lambda_1(\lambda_1-\lambda_q)}
  - \sum_{q >\mu} \tilde{g}^{i\ov{i}} |T^q_{1i}|^2 \frac{(\lambda_1-\lambda_q)}{\lambda_1}   \\
  \ge {} &  - \sum_{q>\mu} \frac{\tilde{g}^{i\ov{i}} |\tilde{g}_{1\ov{q},i}|^2}{\lambda_1(\lambda_1-\lambda_q)} - C\tr{\tilde{g}}{g}.
  \end{split}
  \end{equation}

%For the last term of (\ref{tog}) we use the bound
%\begin{equation} \label{cuatro}
%2\varphi' \re \left(   \tilde{g}^{i\ov{i}} a_{\ov{i}} \left(   u_p u_{i\ov{p}} +   u_{pi} u_{\ov{p}} \right) \right) \ge -(\varphi')^2 \tilde{g}^{i\ov{i}} |u_p u_{i\ov{p}} +u_{pi} u_{\ov{p}}|^2 - C \tr{\tilde{g}}{g}.
%\end{equation}

%Next, compute using  (\ref{Qi}), and picking $\delta$ so that $4\delta \le \frac{\ve}{1-\ve}$ (and hence $\delta \sim 1/A$)
%\[
%\begin{split}
%2\delta  \frac{\tilde{g}^{i\ov{i}} | \tilde{g}_{1\ov{1},i}|^2}{\lambda_1^2} = {} &   2\delta \tilde{g}^{i\ov{i}} \left| \varphi' (  u_p u_{i\ov{p}} +  u_{pi} u_{\ov{p}} ) + \psi' u_i \right|^2 \\
%\le {} & 4\delta (\varphi')^2   \tilde{g}^{i\ov{i}} \left|  u_p u_{i\ov{p}} +  u_{pi} u_{\ov{p}} \right|^2 + 4\delta (\psi')^2   \tilde{g}^{i\ov{i}} |u_i|^2 \\
%\le{} & \frac{\varphi''}{2}   \tilde{g}^{i\ov{i}} \left|  u_p u_{i\ov{p}} + u_{pi} u_{\ov{p}} \right|^2 + \frac{\psi''}{2}  \tilde{g}^{i\ov{i}} |u_i|^2.
%\end{split}
%\]

Combining (\ref{tog}) with (\ref{dontforget}), (\ref{uno}), (\ref{dos}), (\ref{tres}), (\ref{tres2}) and (\ref{tres3})   we obtain for a uniform constant $C$,
%Then from (\ref{tog}) using the Cauchy-Schwarz inequality on the last term, and all of the above, we obtain
\[
\begin{split}
0 \ge {} & \left( \frac{1}{4} \varphi' - \frac{C}{\lambda_1} \right) \sum_{i,p} \tilde{g}^{i\ov{i}} \left( |u_{p\ov{i}}|^2+|u_{pi}|^2 \right) + \left(-\psi'/2 - C \right) \tr{\tilde{g}}{g}  \\ {} & + (\psi'' + C \psi')\tilde{g}^{i\ov{i}}|u_i|^2 + \psi' n.
\end{split}
\]
But since we may assume that $\lambda_1 \ge 4C K$, the first term on the right hand side is nonnegative.
 Pick $A=2(C+1)$ so that $-\psi'/2 -C \ge 1$ and $\psi'' + C\psi' \ge 0$.  Then $\tr{\tilde{g}}{g}$ and hence $\lambda_1$ is uniformly bounded from above at the maximum of $Q$, and the result follows.
\end{proof}

\medskip
\noindent
{\bf Remark.} \
In the proof above we used a viscosity type argument to deal with the non-differentiability of the largest eigenvalue $\lambda_1$.  There are other methods to deal with this issue:  one is to use a perturbation argument as in \cite{Sz, STW}; another is to replace $\lambda_1$ by a carefully chosen quadratic function of $\tilde{g}_{i\ov{j}}$ as in \cite{TW5}.

\section{Proof of the main theorem} \label{sectionproof}

\subsection{Higher order estimates}\label{higher} First, we discuss the {\em a priori} higher order estimates, in the same setting as Theorems \ref{0th} and \ref{2nd}.
Thanks to the estimates in these Theorems, a blowup argument can be employed exactly as in \cite{DK,Sz,STW,TW3} to obtain that
$\sup_M |\de u|_g\leq C,$ and therefore also $\sup_M\tr{g}{\ti{g}}\leq C.$ Here we use the classical Liouville Theorem stating that a bounded plurisubharmonic function on $\mathbb{C}^n$ is constant (indeed, by restricting to complex lines, this reduces to the well-known fact that a bounded subharmonic function in $\mathbb{C}$ is constant).

The PDE \eqref{pde} then implies that $\ti{g}$ is uniformly equivalent to $g$, at which point we can then apply the Evans-Krylov theory \cite{E, K, Tr} (see also \cite{TWWY}) to obtain uniform {\em a priori} $C^{2,\alpha}$ bounds on $u$, for some uniform $0<\alpha<1$. Differentiating the equation and using Schauder theory, we then deduce uniform {\em a priori} $C^k$ bounds for all $k\geq 0.$

\subsection{Existence of a solution} We employ the continuity method. For $t\in [0,1]$ we consider the family of equations for $(u_t,b_t)$
\begin{equation} \label{mat}
\begin{split}
& \det (g_{i\ov{j}} +  a_i u_{t,\ov{j}} + a_{\ov{j}} u_{t,i} + u_{t,i\ov{j}}) = e^{tF+b_t}\det (g_{i\ov{j}}), \\
&  \mathrm{with\ } (g_{i\ov{j}} + a_i u_{t,\ov{j}} + a_{\ov{j}} u_{t,i} + u_{t,i\ov{j}})>0.
\end{split}
\end{equation}
Suppose we have a solution for $t=\hat{t}$ and write $$\hat{\omega}=\omega+\mn a\wedge\db u_{\hat{t}} -\mn\ov{a}\wedge\de u_{\hat{t}}+\ddbar u_{\hat{t}},$$
and $\hat{H}$ for the linearized operator defined as in \eqref{hhat}. By the same argument of Gauduchon \cite{Ga0} that was mentioned earlier, we may find a smooth function $v$,
normalized by $\int_M e^v\hat{\omega}^n=1$, such that
$$\int_M \hat{H}(\psi) e^v\hat{\omega}^n=0,$$
for all smooth functions $\psi$, i.e. $e^v$ generates the kernel of the adjoint $\hat{H}^*$ of $\hat{H}$ (with respect to the $L^2$ inner product with volume form $\hat{\omega}^n$). Fix $0<\alpha<1$ and consider the operator
\[\begin{split}
\Upsilon(\psi)&=\log\frac{(\hat{\omega}+\mn a\wedge\db\psi-\mn\ov{a}\wedge\de\psi+\ddbar\psi)^n}{\hat{\omega}^n}\\
&-\log\left(\int_M e^v(\hat{\omega}+\mn a\wedge\db\psi-\mn\ov{a}\wedge\de\psi+\ddbar\psi)^n\right),
\end{split}\]
mapping $C^{3,\alpha}$ functions $\psi$ with zero average (and such that $\hat{\omega}+\mn a\wedge\db\psi-\mn\ov{a}\wedge\de\psi+\ddbar\psi>0$) to the space of $C^{1,\alpha}$ functions $w$ satisfying $\int_M e^{w+v}\hat{\omega}^n=1$ (whose tangent space at $0$ consists precisely of $C^{1,\alpha}$ functions orthogonal to the kernel of $\hat{H}^*$). For any $C^{3,\alpha}$ function $\zeta$ we have
$$\int_M e^v\hat{H}(\zeta)\hat{\omega}^n=\int_M \zeta\hat{H}^*(e^v)\hat{\omega}^n=0,$$
hence the linearization of $\Upsilon$ at $0$ is $\hat{H}$. Thanks to the Fredholm alternative, $\hat{H}$ is an isomorphism of the tangent spaces, and so the Inverse Function Theorem provides us with $C^{3,\alpha}$ functions $\psi_t$ for $t$ near $\hat{t}$ which satisfy
$$\Upsilon(\psi_t)=(t-\hat{t})F-\log\left(\int_M e^{(t-\hat{t})F}e^v\hat{\omega}^n\right),$$
so that $u_t=u_{\hat{t}}+\psi_t$ solve \eqref{mat} for some $b_t\in\mathbb{R}$. Lastly, differentiating \eqref{mat} and using Schauder estimates and bootstrapping, we easily see that our $C^{3,\alpha}$ solutions are in fact smooth.

This establishes that the set of all $t\in [0,1]$ for which we have a solution $(u_t,b_t)$ of \eqref{mat} is open (and nonempty, since we can take $(u_0,b_0)=(0,0)$). At this point we can also impose that $\sup_Mu_t=0$ by adding a $t$-dependent constant. To show that the set of such $t\in[0,1]$ is also closed, it suffices to prove {\em a priori} estimates for $u_t$ (in $C^k$ for all $k\geq 0$) and $b_t$. The bound $|b_t|\leq \sup_M|F|$ is elementary by the maximum principle, and then the estimates for $u_t$ follow from section \ref{higher} above.

\subsection{Uniqueness} In the setting of the main theorem \ref{main}, uniqueness of $b$ and $u$ follows from a simple maximum principle argument, see e.g. \cite{CTW}.

\end{document}